\documentclass{amsart}
\usepackage{amssymb, amsmath, amscd, latexsym, mathrsfs}
\usepackage{tikz} 
\usetikzlibrary{matrix,arrows} 

\usepackage[all]{xy}

\usepackage{xr-hyper}

\usepackage[pagebackref=true,pdfnewwindow=true,pdftex]{hyperref}
\hypersetup{colorlinks, linkcolor=blue, citecolor=red}

\usepackage{stdmacro}

\newcommand{\Caty}{\mathsf{Cat}_{\infty}}
\newcommand{\Fun}{\mathsf{Fun}}
\newcommand{\SSp}{s\mathcal{S}}

\newcommand{\PSh}{\mathsf{PSh}}

\newcommand{\he}{\textup{he}}

\newcommand{\op}{\textup{op}}

\newcommand{\sW}{\mathcal{W}}

\begin{document}

%% citations
%\nocite{*}

\title{Segal objects and the Grothendieck construction}
\author{Pedro Boavida de Brito}%
\address{Dept. of Mathematics, Instituto Superior Tecnico, Univ. of Lisbon, Av. Rovisco Pais, Lisboa, Portugal}%
\email{pedrobbrito@tecnico.ulisboa.pt}
\thanks{The author is supported by a FCT grant SFRH/BPD/99841/2014.}

%\subjclass{}
%\keywords{}%
%\date{\today}%
%\dedicatory{to ....}%
%\commby{}%
% ----------------------------------------------------------------
\begin{abstract}
We discuss right fibrations in the $\infty$-categorical context of Segal objects in some category $\sV$ and prove some basic results about these.
\end{abstract}
\maketitle

%\setcounter{tocdepth}{1}
%\tableofcontents

% ----------------------------------------------------------------

Right fibrations and Cartesian fibrations (and their dual notions) form the backbone of $\infty$-category theory \cite{Joyal}, \cite{HTT}. The purpose of this note is to investigate them from the point of view of Rezk's Segal spaces and, more generally, Segal objects in some enriching category $\sV$. Due to the known comparison between (complete) Segal spaces and quasi-categories, our arguments also give an alternative route to proving the quasi-categorical analogues, a direct proof of which is often technically more involved. More importantly for us, Segal spaces provide natural models for infinity categories in geometric contexts, and it seems useful to have some theory developed specifically for them.

\medskip
Let us briefly state the main definitions and results. Throughout, $\Sp$ denotes the category of spaces (\emph{space} may be used as a synonym for \emph{simplicial set}). Let $B$ be a Segal space, not necessarily complete (for definitions, see section \ref{sec:gc}). We refer to $B_0$ as the space of objects, $B_1$ as the space of morphisms, $d_0$ is the operator \emph{source} and $d_1$ is the operator \emph{target}. A \textbf{right fibration} over $B$ is a Segal space $X$ together with a map $X \to B $ having the property that the square
\begin{equation}\label{eq:rfibsquare}
	\begin{tikzpicture}[descr/.style={fill=white},baseline=(current bounding box.base)]] 
	\matrix(m)[matrix of math nodes, row sep=2.5em, column sep=2.5em, 
	text height=1.5ex, text depth=0.25ex] 
	{
	X_0 & X_1 \\
	B_0 & B_1 \\
	}; 
	\path[<-,font=\scriptsize] 
		(m-1-1) edge node [auto] {$d_1$} (m-1-2);
	\path[<-,font=\scriptsize] 
		(m-2-1) edge node [auto] {$d_1$} (m-2-2);
	\path[->,font=\scriptsize] 
		(m-1-1) edge node [left] {} (m-2-1);
	\path[->,font=\scriptsize] 		
		(m-1-2) edge node [auto] {} (m-2-2);
	\end{tikzpicture} 
\end{equation}
is homotopy cartesian. A map between right fibrations over $B$ is a weak equivalence if it is a weak  equivalence on $0$-simplices (alias object spaces).

\medskip
 After defining a model structure whose homotopy theory is that of right fibrations, we prove the following strictification result.

\begin{thmA} Let $\sC$ be a category (possibly simplicially enriched or, more generally, internal to spaces).
The Grothendieck construction (see section \ref{sec:gc}) is a simplicial Quillen equivalence between the right fibration model structure on the category $\SSp_{/ N \sC}$ of simplicial spaces over the nerve of $\sC$ and the projective model structure on $\Fun(\sC^{op}, \Sp)$, the category of contravariant simplicial functors from $\sC$ to spaces.
\end{thmA}

The next step in generality are Cartesian fibrations. Still for a fixed Segal space $B$ (which we view as a bisimplicial space with $(m,n)$-simplices $B_m$), we say that a map $X \to B$ of bisimplicial spaces is a \textbf{Cartesian fibration}  if the simplicial spaces $X_{\bullet, n}$ and $X_{m, \bullet}$ are Segal spaces for all $n$ and $m$, and the map $$X_{\bullet,n} \to B_{\bullet,n}$$ is a right fibration for every $n$.
\iffalse\item The square (of simplicial spaces)
\[
	\begin{tikzpicture}[descr/.style={fill=white}] 
	\matrix(m)[matrix of math nodes, row sep=2.5em, column sep=2.5em, 
	text height=1.5ex, text depth=0.25ex] 
	{
	X_{0,\bullet} & X_{1, \bullet} \\
	B_{0, \bullet} & B_{1, \bullet} \\
	}; 
	\path[<-,font=\scriptsize] 
		(m-1-1) edge node [auto] {$d_1$} (m-1-2);
	\path[<-,font=\scriptsize] 
		(m-2-1) edge node [auto] {$d_1$} (m-2-2);
	\path[->,font=\scriptsize] 
		(m-1-1) edge node [left] {} (m-2-1);
	\path[->,font=\scriptsize] 		
		(m-1-2) edge node [auto] {} (m-2-2);
	\end{tikzpicture} 
\]
is homotopy cartesian.
\fi
Moreover, a map $X \to Y$ between Cartesian fibrations is a weak equivalence if is it a Dwyer-Kan equivalence on $0$-simplices $X_{0, \bullet} \to Y_{0, \bullet}$.

\medskip

After defining a model structure whose underlying homotopy theory is that of Cartesian fibrations, we prove the analogous strictification statement below.

\begin{thmB}
The Grothendieck construction is a right Quillen equivalence between the right fibration model structure on $s\SSp_{/ N \sC}$ and the projective model structure on $\Fun(\sC^{op}, \Caty)$. It is moreover an equivalence of $(\infty,2)$-categories.
\end{thmB}

Here $\Caty$ refers to a category of $(\infty,1)$-categories, which for concreteness we take to be Rezk's model category of complete Segal spaces \cite{Rezk}.

\medskip

The definition of a right fibration of (complete) Segal spaces is not new nor unexpected; Michael Shulman has mentioned a similar definition in unpublished notes of Charles Rezk, and Ilan Barnea pointed out to us that a definition has appeared (independently) in \cite{Kazhdan}. The definition of a Cartesian fibration above has a different flavor to the construction that usual bears that name (in categories or quasi-categories). Theorem B argues that from a certain point of view both try to achieve the same (more on this in remark \ref{rem:cartfib}).

\medskip
Theorems A and B have counterparts in the quasi-categorical setting that have appeared in the work of Joyal \cite{Joyal}, Lurie \cite{HTT} and Heuts-Moerdjik \cite{HeutsMoerdijk1}, \cite{HeutsMoerdijk2}. An exception is the internal version of these theorems, as far as we know. Rezk's theory of complete Segal spaces is manifestly under-developed when compared to that of quasi-categories, and we hope this fact is enough justification for this paper. Moreover, the approach used here seems well suited to generalizations, for example to the $(\infty,n)$ setting.

\subsection*{Outline} In section \ref{sec:gc}, after some preliminary work, we prove theorem A. The key ingredient is a version of the Yoneda lemma (lemma \ref{lem:yoneda}). We also show that right fibrations are fibrations in the \emph{complete} Segal space model structure, and directly relate our definition of a right fibration to the quasi-categorical one. In section \ref{sec:general} we define the homotopy theory underlying right fibration objects in a general category $\sV$, and in section \ref{sec:cartfib} we apply that to define the homotopy theory of Cartesian fibrations and prove theorem B. In section \ref{sec:colim} we describe (co)limits of right fibrations. In section \ref{sec:kan} we give an explicit formula for Kan extensions of a right fibration in spaces and then show that the homotopy theory of right fibrations is invariant under (Dwyer-Kan) equivalences of the base.

\subsection*{Notation} 
As customary, we denote the simplex category by $\Delta$. Objects in $\Delta$ are finite totally ordered sets $[n]= \{ 0 < 1 < \dots < n\}$, for $n \geq 0$, and morphisms are order-preserving maps.
Often we implicitly view $[n]$ as a category $0 \to 1 \to \dots \to n$.

 We write $s\Sp$ for the category of simplicial spaces, or more generally, $s\sV$ for the category of simplicial objects in a category $\sV$, that is, contravariant functors from $\Delta$ to $\sV$. We frequently and unapologetically regard a simplicial set $X$ as a simplicial \emph{discrete} space with $n$-simplices given by the constant simplicial set $X_n$. In particular, we keep the notation $\Delta[n]$ for the simplicial discrete space (which Rezk refers to as $F[n]$). For a fixed object $C$ in $s\sV$, we write $s\sV_{/C}$ for the category of simplicial objects over $B$. 

 For a category $\sC$, which may be enriched or internal in spaces, we write $N\sC$ for the nerve of $\sC$, i.e. the simplicial \emph{space} whose space of $0$-simplices is the space $\textup{ob}(C)$ of objects of $\sC$, the space of $1$-simplices is the space $\textup{mor}(\sC)$ of morphisms of $\sC$ and whose space of $n$-simplices is the space of $n$ composable morphisms 
 $$c_0 \gets \dots \gets c_n$$ of $\sC$ (which is expressed as an iterated pullback). Applied to such a string, the face map $d_i$ drops $c_{i}$ and the degeneracy map $s_i$ replaces $c_{i}$ by an identity map. With these conventions, we follow Bousfield-Kan \cite{BousfieldKan}, but we warn the reader that these differ from other sources (e.g. Rezk).

\subsection*{Acknowledgments}
Thanks to Geoffroy Horel and Mike Shulman for discussions and encouragement; it was their suggestion that we upgrade the result to internal categories and include the relation with quasi-categories. The author is also grateful to the referee for a careful reading and helpful comments. An earlier iteration of this paper appeared in the author's PhD thesis.

\section{Left and Right fibrations}\label{sec:gc}

\subsection{Segal spaces}

The following definition is due to Charles Rezk. It is a model for the umbrella term $\infty$-\emph{category}.

\begin{defn}[\cite{Rezk}]
A \textbf{Segal space} is a simplicial space $X$ such that, for each $n \ge 2$, the map 
\begin{equation}\label{eq:segalmaps}
X_n \xrightarrow{(\alpha_0^*, \dots, \alpha_{n-1}^*)} \holim (X_1 \xrightarrow{d_0} X_0 \xleftarrow{d_1} X_1 \xrightarrow{d_0} \dots  \xleftarrow{d_1} X_1 )
\end{equation}
is a weak equivalence of spaces, where $\alpha_i$ is the map $\Delta[1] \to \Delta[n]$ sending $0 < 1$ to $i < i+1$.
\end{defn}

Elements of $X_0$ are called \emph{objects} and elements of $X_1$ are called \emph{morphisms}.

\begin{expl}\label{ex:segal}
The nerve of an internal category in spaces is a Segal space if either the source or target map is a fibration, using that the category of simplicial sets is right proper. 
The nerve of a simplicially enriched category is always a Segal space. Indeed, it satisfies the non-derived version of (\ref{eq:segalmaps}) and we may assume that each space $X_i$ is fibrant (for example by applying Kan's $\textup{Ex}^\infty$ functor which commutes with finite limits). Then the source-target operators map onto a discrete space and so are fibrations and in that case the relevant pullbacks are homotopy pullbacks.
\end{expl}

\begin{defn}\label{defn:mapspace}
Let $X$ be a Segal space. Given two objects $x, y \in X_0$, the space of morphisms $\textup{mor}_X(x,y)$ is defined as the homotopy fiber, over $(x,y)$, of
\[
X_1 \xrightarrow{(d_0, d_1)} X_0 \times X_0 \; .
\]
Two maps $f, g \in \textup{mor}_X(x,y)$ are said to be homotopic if they lie in the same component. The \emph{homotopy category} of $X$ is the category $\mathsf{Ho}(X)$ with objects $X_0$ and morphisms $\map_{\mathsf{Ho}(X)}(x,y) := \pi_0 \textup{mor}_X(x,y)$. A map $f \in \textup{mor}_X(x,y)$ is a \emph{homotopy equivalence} if it admits a left and right inverse in $\mathsf{Ho}(X)$. 
\end{defn}

The space of homotopy equivalences is denoted $X_1^{\he}$. It is necessarily a union of components of $X_1$ since if a map can be connected by a path to a homotopy equivalence, then it must itself be a homotopy equivalence \cite[lemma 5.8]{Rezk}.

\begin{defn}\label{defn:DKcat}
A simplicial map $i: C \to D$ between Segal spaces is a \emph{Dwyer-Kan equivalence} if it is fully faithful, i.e. for each pair of objects $x, y \in C_0$, the induced map
\[
\textup{mor}_C(x,y) \rightarrow \textup{mor}_D(f(x), f(y))
\]
 is a weak equivalence of spaces, and $\mathsf{Ho}(f)\co \mathsf{Ho}(X) \rightarrow \mathsf{Ho}(Y)$ is essentially surjective.
\end{defn}

\begin{expl}
Suppose $(\sC,\mathcal{W})$ is a category $\sC$ with a subcategory $\mathcal{W}$ containing all objects. Such a pair $(\sC,\mathcal{W})$ goes by the name of \emph{relative category} in \cite{BarwickKan1}. From a relative category $(\sC, \sW)$, Rezk defines a Segal space $N(\sC,\sW)$, the \emph{Rezk nerve} (Rezk calls it the \emph{classification diagram}) which retains the homotopical information of the pair. The simplicial set $N(\sC, \sW)_m$ of $m$-simplices is defined as the nerve of the category of contravariant functors $[n] \times [m] \to \sC$ which send all the arrows in the $n$-direction to $\sW$.

 An older, closely related construction is the \emph{hammock localization} of Dwyer-Kan \cite{DwyerKan}. This is a simplicially enriched category $L^H\sC$ which is a simplicial enhancement of the classical localization, in the sense that the category obtained by taking $\pi_0$ of the morphism spaces in $L^H\sC$ is the classical localization of $\sC$ with respect to $\mathcal{W}$. If $\sC$ happens to be a simplicially enriched model category, then the morphism spaces in $L^H \sC$ are identified with the derived mapping spaces computed relative to the model structure on $\sC$. 

 In \cite{BarwickKan2}, Barwick and Kan show that the Rezk nerve and the usual nerve of the hammock localization agree. More precisely, they show that $N(L^H \sC)$, the nerve of $L^H \sC$, is naturally weakly equivalent (in the complete Segal space sense of \cite{Rezk}) to $N(\sC,\sW)$\footnote{In more detail, Barwick-Kan show that the Rezk nerve of $(\sC,\sW)$ is weakly equivalent (in the model structure for complete Segal spaces) to the Rezk nerve of the relativization of $L^H \sC$ \cite[proposition 3.1]{BarwickKan2}, and in \cite[proposition 1.11]{BarwickKan2} that the Rezk nerve of the relativization of a simplicially enriched category $A$ is naturally weakly equivalent to $N(A)$, the (usual) nerve of $A$. (Warning: Barwick-Kan use the letter $N$ for the Rezk nerve, and the letter $Z$ for what we refer to as the nerve.)}. A map between Segal spaces is a complete Segal weak equivalence if and only if it is a Dwyer-Kan equivalence. So if $N(\sC,\sW)$ is a Segal space (e.g. if $\sC$ has a model category structure with weak equivalences given by $\sW$) then $N(\sC,\sW)$ is Dwyer-Kan equivalent to $N(L^H \sC)$.
\end{expl}

\subsection{Right fibrations}

Let $F$ be a \emph{contravariant} functor from $\sC$ to spaces (possibly enriched). The \emph{Grothendieck construction} of $F$ is a simplicial space $\sG(F)$ over $N\sC$ defined as
\[
\sG(F)_n := \coprod_{c_0, \dots, c_n} \map_{\sC}(c_n, c_{n-1}) \times \dots \times \map_{\sC}(c_{1}, c_0) \times F(c_0)
\]
Degeneracies select the identity maps. Face maps encode functoriality. For example, the maps $d_0$ and $d_1$ 
\[
\coprod_{c_0, c_1} \map_{\sC}(c_1, c_0) \times F(c_0) \rightarrow \coprod_{c_0} F(c_0)
\]
are composition and projection, respectively (remember $F$ is contravariant). As defined, $\sG(F)$ is the nerve of an internal category in spaces (i.e. a category object in spaces, which we may assume are all fibrant as in Example \ref{ex:segal}) and the target map is a projection, hence is a fibration. So $\sG(F)$ is a Segal space. There is reference functor from $\sG(F)$ to the nerve of $\sC$ and it is easy to see that $\sG$ is fully faithful as a functor $\Fun(\sC^{op}, \Sp) \to \SSp_{/N\sC}$. Clearly, the Grothendieck construction $\sG(F) \to N\sC$ of a \emph{contravariant} functor $F$ is a right fibration.

\begin{rem}
The dual notion to a right fibration is a \emph{left fibration}. A map $p : X \to B$ between Segal spaces is a left fibration if $X^{op} \to B^{op}$ is a right fibration\footnote{By definition, the opposite of a Segal space $Y$ is the Segal space $Y^{op}$ obtained by reversing the order of the simplicial maps, i.e. $Y^{op}_{n} := Y_n$, $d_{i} : Y_n^{op} \to Y_{n-1}^{op}$ is defined as $d_{n-i} : Y_n \to Y_{n-1}$ and similarly for $s_i$.}. Alternatively, $p$ is a left fibration if the variant of the square (\ref{eq:rfibsquare}) with $d_1$ replaced by $d_0$, is homotopy cartesian. The Grothendieck construction of a \emph{covariant} functor is a \emph{left} fibration.
\end{rem}

\begin{prop}\label{prop:rightfibequiv}
Suppose that the square (\ref{eq:rfibsquare}) is homotopy cartesian for a map of simplicial spaces $X \to B$ and $B$ is a Segal space. The following statements are equivalent.
\begin{enumerate}
\item $X$ is a Segal space (i.e. $X \to B$ is a right fibration)
\item For each $n \ge 1$, the diagram
\[
	\begin{tikzpicture}[descr/.style={fill=white}, baseline=(current bounding box.base)] ]
		\matrix(m)[matrix of math nodes, row sep=2.5em, column sep=2.5em, 
	text height=1.5ex, text depth=0.25ex] 
	{
	X_0 & X_n \\
	B_0 & B_n \\
	}; 
	\path[<-,font=\scriptsize] 
		(m-1-1) edge node [auto] {$a_n^*$} (m-1-2);
	\path[<-,font=\scriptsize] 
		(m-2-1) edge node [auto] {$a_n^*$} (m-2-2);
	\path[->,font=\scriptsize] 
		(m-1-1) edge node [left] {} (m-2-1);
	\path[->,font=\scriptsize] 		
		(m-1-2) edge node [auto] {} (m-2-2);
	\end{tikzpicture} 
\]
is homotopy cartesian, where $a_{n} : [0] \rightarrow [n]$ is the map sending $0$ to $0$.
\item For each $n \ge 1$, the diagram
\[
	\begin{tikzpicture}[descr/.style={fill=white}] 
	\matrix(m)[matrix of math nodes, row sep=2.5em, column sep=2.5em, 
	text height=1.5ex, text depth=0.25ex] 
	{
	X_{n-1} & X_n \\
	B_{n-1} & B_n \\
	}; 
	\path[<-,font=\scriptsize] 
		(m-1-1) edge node [auto] {$d_n$} (m-1-2);
	\path[<-,font=\scriptsize] 
		(m-2-1) edge node [auto] {$d_n$} (m-2-2);
	\path[->,font=\scriptsize] 
		(m-1-1) edge node [left] {} (m-2-1);
	\path[->,font=\scriptsize] 		
		(m-1-2) edge node [auto] {} (m-2-2);
	\end{tikzpicture} 
\]
is homotopy cartesian, where $d_n$ is induced by the face map $d^n : [n-1] \rightarrow [n]$ which misses $n$.
\end{enumerate}
\end{prop}
\begin{proof}
This is rather formal and it works for any model category which, like $\Sp$, is right proper. Suppose $(1)$ holds. We show $(2)$ holds by induction on $n$. The case $n=1$ holds by assumption. Suppose, for the induction hypothesis, that $(2)$ holds for $n-1$. Factor the square in $(2)$ as in the diagram below.
\[
	\begin{tikzpicture}[descr/.style={fill=white}] 
	\matrix(m)[matrix of math nodes, row sep=2.5em, column sep=2.5em, 
	text height=1.5ex, text depth=0.25ex] 
	{
	X_{0} & {X_{n-1}} {\times^h_{X_0}} X_1 & X_{n} \\
	B_{0} & {B_{n-1}} {\times^h_{B_0}} B_1 & B_{n} \\
	}; 
	\path[<-,font=\scriptsize] 
		(m-1-1) edge node [auto] {} (m-1-2);
	\path[<-,font=\scriptsize] 
		(m-2-1) edge node [auto] {} (m-2-2);
	\path[->,font=\scriptsize] 
		(m-1-1) edge node [left] {} (m-2-1);
	\path[->,font=\scriptsize] 		
		(m-1-2) edge node [auto] {} (m-2-2);
	\path[<-,font=\scriptsize] 
		(m-1-2) edge node [left] {} (m-1-3);
	\path[->,font=\scriptsize] 		
		(m-1-3) edge node [auto] {} (m-2-3);
	\path[<-,font=\scriptsize] 		
		(m-2-2) edge node [auto] {} (m-2-3);
	\end{tikzpicture} 
\]
(rows induced by $\Delta[0] \to \Delta[n-1] \amalg_{\Delta[0]}\Delta[1] \to \Delta[n]$; the first map is given by the inclusion $\alpha_1 = d^1 : \Delta[0] \to \Delta[1]$, the second is the Segal map)

 The right-hand square is homotopy cartesian since the horizontal maps are weak equivalences (using the assumption that $X$ and $B$ are Segal spaces). The left-hand square is also homotopy cartesian; this follows by commuting homotopy pullbacks, the induction hypothesis and the assumption that $X \to B$ is a right fibration.

 The equivalence between $(2)$ and $(3)$ is straightforward.

 It remains to show that $(3)$ implies $(1)$. We will verify the Segal condition by showing that the commutative square
\begin{equation}\label{eq:segalagain}
	\begin{tikzpicture}[descr/.style={fill=white}, baseline=(current bounding box.base)] ]
	\matrix(m)[matrix of math nodes, row sep=2.5em, column sep=2.5em, 
	text height=1.5ex, text depth=0.25ex] 
	{
	X_n & X_{n-1} \\
	X_{n-1} & X_{n-2} \\
	}; 
	\path[->,font=\scriptsize] 
		(m-1-1) edge node [auto] {$d_0$} (m-1-2);
	\path[->,font=\scriptsize] 
		(m-2-1) edge node [auto] {$d_{0}$} (m-2-2);
	\path[->,font=\scriptsize] 
		(m-1-1) edge node [left] {$d_n$} (m-2-1);
	\path[->,font=\scriptsize] 		
		(m-1-2) edge node [auto] {$d_{n-1}$} (m-2-2);
	\end{tikzpicture} 
\end{equation}
is homotopy cartesian, for each $n \ge 2$. It is not hard to deduce that this is an equivalent formulation of Segal's condition. Now prolong the above square to the diagram
\[
	\begin{tikzpicture}[descr/.style={fill=white}] 
	\matrix(m)[matrix of math nodes, row sep=2.5em, column sep=2.5em, 
	text height=1.5ex, text depth=0.25ex] 
	{
	 X_{n}  & {X_{n-1}} & {B_{n-1}}\\
	 X_{n-1} & {X_{n-2}} & {B_{n-2}} \\
	}; 
	\path[->,font=\scriptsize] 
		(m-1-1) edge node [auto] {$d_0$} (m-1-2);
	\path[->,font=\scriptsize] 
		(m-2-1) edge node [auto] {$d_{0}$} (m-2-2);
	\path[->,font=\scriptsize] 
		(m-1-1) edge node [left] {$d_n$} (m-2-1);
	\path[->,font=\scriptsize] 		
		(m-1-2) edge node [auto] {$d_{n-1}$} (m-2-2);
	\path[->,font=\scriptsize] 
		(m-1-2) edge node [left] {} (m-1-3);
	\path[->,font=\scriptsize] 		
		(m-1-3) edge node [auto] {$d_{n-1}$} (m-2-3);
	\path[->,font=\scriptsize] 		
		(m-2-2) edge node [auto] {} (m-2-3);
	\end{tikzpicture} 
\]
Here, the right-hand square is homotopy cartesian by assumption. So it remains to show that the outer rectangle is homotopy cartesian. That outer rectangle is identified with
the outer rectangle in the diagram:
\[
	\begin{tikzpicture}[descr/.style={fill=white}] 
	\matrix(m)[matrix of math nodes, row sep=2.5em, column sep=2.5em, 
	text height=1.5ex, text depth=0.25ex] 
	{
	 X_{n}  & {B_{n}} & {B_{n-1}}\\
	 X_{n-1} & {B_{n-1}} & {B_{n-2}} \\
	}; 
	\path[->,font=\scriptsize] 
		(m-1-1) edge node [auto] {} (m-1-2);
	\path[->,font=\scriptsize] 
		(m-2-1) edge node [auto] {} (m-2-2);
	\path[->,font=\scriptsize] 
		(m-1-1) edge node [left] {$d_n$} (m-2-1);
	\path[->,font=\scriptsize] 		
		(m-1-2) edge node [auto] {$d_n$} (m-2-2);
	\path[->,font=\scriptsize] 
		(m-1-2) edge node [auto] {$d_0$} (m-1-3);
	\path[->,font=\scriptsize] 		
		(m-1-3) edge node [auto] {$d_{n-1}$} (m-2-3);
	\path[->,font=\scriptsize] 		
		(m-2-2) edge node [auto] {$d_{0}$} (m-2-3);
	\end{tikzpicture} 
\]
But both squares in this diagram are homotopy cartesian, the left-hand square by assumption $(2)$ and the right-hand square since $B$ is a Segal space.
\end{proof}

\subsection{Homotopy theory of right fibrations}

We begin by recalling the Segal space model structure on the category of simplicial spaces. There is, throughout the paper, an underlying model structure on the category of simplicial spaces with \emph{levelwise} weak equivalences from which all the other model structures are obtained by localization. There are at least two choices for this model structure, depending on whether we ask the cofibrations or the fibrations to be given levelwise. It is well known that the identity map gives a Quillen equivalence between the two model structures (usually called injective and projective model structure, respectively). Because of this, we refer to either generically as the \emph{levelwise} model structure. So, whenever we say that a map is a weak equivalence, fibration or cofibration of \emph{simplicial spaces} or \emph{in the levelwise model structure}, we mean that in either of the two senses above (but the meaning should of course be fixed, at least locally).

\begin{prop}\label{prop:rfibmodel}(Rezk)
There is a left proper, simplicial model structure structure on $\SSp$ -- called the \textbf{Segal space model structure} -- uniquely characterized by the following properties:
\begin{itemize}
\item An object $X$ is fibrant if it is a Segal space (and is fibrant as a simplicial space).
\item A morphism $X \to Y$ between (fibrant) Segal spaces is a weak equivalence if $X_n \to Y_n$ is a weak equivalence for each $n \geq 0$.
\item A morphism is a cofibration if it is a cofibration of simplicial spaces.
\end{itemize}
More generally, a morphism $f : X \to Y$ between any two simplicial spaces is a weak equivalence if the induced map on derived mapping spaces
\[
\RR \map(f, Z) : \RR \map(Y, Z) \to \RR \map(X, Z)
\]
(computed with respect to the levelwise model structure) is a weak equivalence of spaces for every Segal space $Z$.
\end{prop}

\begin{rem}(Alternative characterization of the Segal condition)
For each $n \geq 2$, the map
\[
(\alpha_0, \dots, \alpha_{n-1}) : \Delta[1] \amalg_{\Delta[0]} \dots \amalg_{\Delta[0]} \Delta[1] \hookrightarrow \Delta[n]
\]
is a weak equivalence by definition of the Segal space model structure. Here is a folklore statement that resonates with the definition of a quasi-category. For $n \geq 2$, let $\iota: \Lambda^k[n] \hookrightarrow \Delta[n]$ be an inner horn inclusion, i.e. $0 < k < n$. Then $\iota$ is a weak equivalence in the Segal space model structure. One can show this inductively on $n$. The case $n  = 2$ is already part of the Segal condition. For higher $n$, factor the map $(\alpha_0, \dots, \alpha_{n-1})$ as
\[
\Delta[1] \amalg_{\Delta[0]} \dots \amalg_{\Delta[0]} \Delta[1]  \hookrightarrow \Lambda^k[n] \to \Delta[n]
\]
The left-hand map can be given as a composition of inner horn extensions, i.e. pushouts along inner horn inclusions. These inner horn inclusions are all of \emph{lower} dimension and so are weak equivalences by induction. It follows that each extension is a pushout along a weak equivalence (which is also a cofibration in the injective structure) and so a weak equivalence. By two-out-of-three, we conclude that the right-hand map in the display is a weak equivalence.
\end{rem}

\begin{prop}\label{prop:rfibmodel}
Fix a Segal space $B$. There is a left proper, simplicial model structure structure on $\SSp_{/B}$ -- called the \textbf{right fibration model structure} -- uniquely characterized by the following properties:
\begin{itemize}
\item An object $X \to B$ is fibrant if it is a right fibration and a fibration of simplicial spaces.
\item A morphism $X \to Y$ between right fibrations over $B$ is a weak equivalence if $X_0 \to Y_0$ is a weak equivalence, i.e. if the maps of homotopy fibers $X_b \to Y_b$ are weak equivalences for each $b \in B_0$.
\item A morphism is a cofibration if it is a cofibration of simplicial spaces.
\end{itemize}
More generally, a morphism $f : X \to Y$ between any two simplicial spaces over $B$ is a weak equivalence if the induced map on derived mapping spaces
\[
\RR \map_B(f, W) : \RR \map_B(Y, W) \to \RR \map_B(X, W)
\]
(computed with respect to the levelwise model structure) is a weak equivalence of spaces for every fibrant object $W \to B$.
\end{prop}
\begin{proof} Start with the Segal space model structure on simplicial spaces. Take the overcategory model structure. Note that a fibrant object is a fibration in the Segal space model structure (which, since $B$ is fibrant there, is equivalent to $X$ being Segal and the map being a levelwise fibration). Then take the left Bousfield localization with respect to the set of maps $$\{ \LL \Delta[0] \xrightarrow{\LL d^1} \LL \Delta[1] \xrightarrow{f} B : f \in B_1\}$$
where $\LL d^1$ refers to a replacement of $d^1$ by a cofibration in the levelwise model structure (note that no replacement is needed for the injective model structure.)
\end{proof}

\begin{prop}\label{prop:rfibbase}
Right fibrations are stable under base change.
\end{prop}
\begin{proof}
Suppose $X \to B$ is a right fibration and $B^{\prime} \to B$ a map. Let $X^\prime$ denote the homotopy pullback $X \times^h_B B^\prime$. We have a commutative diagram
\begin{equation*}
	\begin{tikzpicture}[descr/.style={fill=white}, baseline=(current bounding box.base)] ]
	\matrix(m)[matrix of math nodes, row sep=2.5em, column sep=2.5em, 
	text height=1.5ex, text depth=0.25ex] 
	{
	X^\prime_1 & X^\prime_0 & X_0\\
	B^\prime_1 & B^\prime_0 & B_0\\
	}; 
	\path[->,font=\scriptsize] 
		(m-1-1) edge node [auto] {$d_1$} (m-1-2);
	\path[->,font=\scriptsize] 
		(m-2-1) edge node [auto] {$d_1$} (m-2-2);
	\path[->,font=\scriptsize] 
		(m-1-2) edge node [left] {} (m-1-3);
	\path[->,font=\scriptsize] 		
		(m-2-2) edge node [auto] {} (m-2-3);
	\path[->,font=\scriptsize] 		
		(m-1-3) edge node [auto] {} (m-2-3);
	\path[->,font=\scriptsize] 
		(m-1-1) edge node [left] {} (m-2-1);
	\path[->,font=\scriptsize] 		
		(m-1-2) edge node [auto] {} (m-2-2);
	\end{tikzpicture} 
\end{equation*} 
The right-hand square is homotopy cartesian by assumption. The outer rectangle is also homotopy cartesian, since it is also given as the composition of two different homotopy cartesian squares. So the left-hand square is homotopy cartesian.
\end{proof}

\begin{prop}
Suppose $f : X \to Y$ and $g : Y \to Z$ are maps between Segal spaces. 
\begin{enumerate}
\item $f$ and $g$ are right fibrations $\Rightarrow$ $gf$ is a right fibration
\item $g$ and $gf$ are right fibrations $\Rightarrow$ $f$ is a right fibration
\item $f$ and $gf$ are right fibrations and $f$ induces a surjection $\pi_0 X_0 \to \pi_0 Y_0$ $\Rightarrow$ $g$ is a right fibration.
\end{enumerate}
\end{prop}
\begin{proof}
$(1)$ and $(2)$ are straightforward. $(3)$ follows from \cite[lemma B.9]{paper2}.
\end{proof}

\begin{rem} Proposition \ref{prop:rightfibequiv} can be rephrased as follows. For every string of morphisms $b = (b_0 \gets \dots \gets b_{n}) \in B_n$, the maps 
\[
\Delta[0] \xrightarrow{\alpha_n} \Delta[n] \xrightarrow{b} B
\quad \quad \textup{ and } \quad \quad
\Delta[n-1] \xrightarrow{d^n} \Delta[n] \xrightarrow{b} B
\]
are weak equivalences in the right fibration model structure.
\end{rem}

\subsection{Right fibrations are fibrations in Rezk's model structure of complete Segal spaces}\label{sec:fibcss}

\begin{defn}
A simplicial space $X$ is \textbf{complete} if either (and therefore all) of the maps $d_0, d_1, s_0$ induce a weak equivalence $X_1^{\he} \simeq X_0$.
\end{defn}

\begin{rem}\label{rem:E}
The following reformulation of the completeness condition is useful. Let $E$ denote the (nerve of the) groupoid with two objects $x,y$ and exactly two (non-identity) isomorphisms $x \to y$ and $y \to x$. A simplicial space $X$ is complete if and only if the map
\[
\RR \map(E,X) \to \RR \map(\Delta[0], X)
\]
induced by any of the (two) maps $\Delta[0] \to E$, is a weak equivalence. The \emph{complete} Segal space model structure is constructed by localizing the Segal space model structure with respect to these maps. The weak equivalences between Segal spaces (not necessarily complete) become the Dwyer-Kan equivalences.
\end{rem}

\begin{defn}
A simplicial space $X$ over $B$ is \textbf{fiberwise complete} if the square
\[
	\begin{tikzpicture}[descr/.style={fill=white}] 
	\matrix(m)[matrix of math nodes, row sep=2.5em, column sep=2.5em, 
	text height=1.5ex, text depth=0.25ex] 
	{
	X_0 & X_1^{\he} \\
	B_0 & B_1^{\he}\\
	}; 
	\path[->,font=\scriptsize] 
		(m-1-2) edge node [above] {$d_1$} (m-1-1);
	\path[->,font=\scriptsize] 
		(m-2-2) edge node [above] {$d_1$} (m-2-1);
	\path[->,font=\scriptsize] 
		(m-1-1) edge node [left] {} (m-2-1);
	\path[->,font=\scriptsize] 		
		(m-1-2) edge node [auto] {} (m-2-2);
	\end{tikzpicture} 
\]
is homotopy cartesian. 
\end{defn}

\begin{rem}
A levelwise fibration $p : X \to B$ is a fiberwise complete Segal space if and only if every diagram
\begin{equation}\label{eq:fc}
	\begin{tikzpicture}[descr/.style={fill=white}, baseline=(current bounding box.base)] ]
	\matrix(m)[matrix of math nodes, row sep=2.5em, column sep=2.5em, 
	text height=1.5ex, text depth=0.25ex] 
	{
	\Delta[0] & X \\
	E & B\\
	}; 
	\path[->,font=\scriptsize] 
		(m-1-1) edge node [auto] {} (m-1-2);
	\path[->,font=\scriptsize] 
		(m-2-1) edge node [auto] {} (m-2-2);
	\path[->,font=\scriptsize] 
		(m-1-1) edge node [left] {target} (m-2-1);
	\path[->,font=\scriptsize] 		
		(m-1-2) edge node [auto] {} (m-2-2);
	\end{tikzpicture} 
\end{equation}
has a contractible space of lifts for each pair of horizontal maps. More precisely, for each $x \in B_1^{\he}$, the map
\begin{equation}\label{eq:fcmap}
\RR \map_{B}(f, X) :  \RR \map_{B}(E, X) \xrightarrow{} \RR \map_{B}(\Delta[0], X)
\end{equation}
induced by 
$
\Delta[0] \rightarrow E \xrightarrow{x} B
$
is a weak equivalence of spaces.
\end{rem}

\begin{prop}\label{prop:fibcss}
Let $B$ be a Segal space. A map $p: X \to B$ is a fibration in the complete Segal space model structure if and only if $p$ is fiberwise complete and a levelwise fibration. In particular, a map which is a right fibration and a levelwise fibration is a fibration in the complete Segal space model structure.
\end{prop}

\begin{proof}
The first assertion is shown in \cite[Appendix B]{paper2}. To prove the second assertion, we suppose $X \to B$ is a right fibration (and a levelwise fibration) and will show that $X$ is fiberwise complete over $B$. Right fibrations are conservative; that is, if a morphism $f$ in $X$ is mapped to an identity, then $f$ is an identity. So if $g : b^\prime \to b$ and $h : b \to b^\prime$ are homotopy inverse in $B$ and $x \in X_0$ mapping to $b$, then by the right fibration condition these lift uniquely to a pair of morphisms $\overline{g}$ and $\overline{h}$ in $X$, which by conservativity must be homotopy inverse.
\end{proof}

\begin{cor}
If $X \to B$ is a right fibration and $B$ is a complete Segal space, then $X$ is a complete Segal space.
\end{cor}

\begin{rem}\label{rem:fibcss}
Proposition \ref{prop:fibcss} implies that the left Bousfield localization of the (overcategory) \emph{complete} Segal space model structure with respect to the same maps as in proposition \ref{prop:rfibmodel} defines the same model category.
\end{rem}

\begin{prop}
Suppose $f : X \to Y$ is a Dwyer-Kan equivalence between Segal spaces over $B$. Then $f$ is a weak equivalence in the right fibration model structure. If in addition $f$ is a right fibration then $f$ is a levelwise weak equivalence.
\end{prop}
\begin{proof}
Dwyer-Kan equivalences are the weak equivalences between Segal spaces in the complete Segal space model structure, and the right fibration model structure on $B$ is a left Bousfield localization of the overcategory of the complete Segal space model structure, so the first claim follows.

 As for the second claim, using proposition \ref{prop:fibcss} the assumption is that $f$ is an acyclic fibration in the complete Segal space model structure. But in any left Bousfield localization, the set of acyclic fibrations coincides with the set of acyclic fibrations of the underlying model structure (which in this case is the levelwise model structure on simplicial spaces).
\end{proof}

\subsection{Comparison with quasi-categories}\label{sec:qcats}
Recall that a space (simplicial set) $X$ is viewed as a simplicial discrete space (bisimplicial set) by taking the pullback along $p_1 : \Delta \times \Delta \to \Delta$, the projection onto the first factor. (The space (set) of $n$-simplices of $p_1^*X$ is $X_n$, viewed as a constant simplicial set.) The functor $p_1$ has a left adjoint, $i_1$, given by $i_1([n]) = ([n], [0])$, so there is an adjunction
\[
p_1^* : \Sp \leftrightarrows s\Sp : i_1^*X
\]
The set of $n$-simplices of $i_1^*X$ is the set of $0$-simplices of $X_n$, the space of $n$-simplices of $X$. By  \cite{JoyalTierney}, the pair $(p_1^*, i_1^*)$ forms a Quillen equivalence between the model structure for quasi-categories (alias the \emph{Joyal} model structure) and the model structure for complete Segal spaces (with underlying Reedy model structure on simplicial spaces). We can use that to prove that the notion of right fibration for Segal spaces agrees with the quasi-categorical one:

\begin{thm} Let $B$ be a (Reedy fibrant) complete Segal space. Then $(i_1^*,p_1^*)$ forms a Quillen equivalence between $\SSp_{/B}$ and $\Sp_{/i_1^*B}$, both equipped with the right fibration model structure.
\end{thm}
\begin{proof}
According to \cite{HeutsMoerdijk1}, when $Y$ is a quasi-category, the right fibration model structure on $\Sp_{/Y}$ is obtained by left Bousfield localization of the overcategory of the Joyal model structure, with respect to right horn inclusions, i.e. maps $\Lambda^n[0] \hookrightarrow \Delta[n] \to Y$ for each $n \geq 1 $, where $\Lambda^n[0]$ is the union of all the faces except $d^0$. (Note that Moerdijk-Heuts'  use different conventions, so their definitions are the mirror of ours.) For a complete Segal space $B$, we have established (c.f. remark \ref{rem:fibcss}) that the right fibration model structure on $\SSp_{/B}$ can be obtained as the left Bousfield localization of the overcategory of the complete Segal space model structure with respect to the maps $p_1^*\Delta[0] \xrightarrow{} p_1^*\Delta[1] \to Y$ (left-hand map is $d^1$).

 The claim is that the overcategory model structures are being localized at the same set of maps. In other words, the maps $p_1^*(\Lambda^n[n]) \hookrightarrow p_1^*(\Delta[n])$ (over $\LL p_1^*i_1^*B \simeq B$) are weak equivalences in the complete Segal space side. For $n = 1$ there is nothing to prove. For higher $n$, factor $d^n$ as
\[
\Delta[n-1] \to \Lambda^n[0] \to \Delta[n]
\]
(Here, and in line with the rest of the paper, we have omitted $p_1^*$ from the notation.)
The composite map is a weak equivalence by proposition \ref{prop:fibcss}. The left-hand map can be
given as a composition of lower dimensional right horn extensions (that is, pushouts along right horn inclusions $\Lambda^k[0] \to \Delta[k]$ for $k < n$), each of which is an acyclic cofibration by induction. Therefore, the left-hand map in the display is a weak equivalence and so, by two-out-of-three, the right-hand map is a weak equivalence.
\end{proof}

\begin{rem}
The condition that $B$ is \emph{complete} can be dropped. Indeed, a Segal space $B$ which is not necessarily complete can be completed: in \cite{Rezk}, Rezk constructs a Dwyer-Kan equivalence $B \to B^{\sharp}$ to a \emph{complete} Segal space $B^\sharp$. The homotopy theories of right fibrations over $B$ and $B^{\sharp}$ are equivalent (this will be proved in corollary \ref{cor:comp}). Therefore, by the theorem above, the homotopy theory of right fibrations over $B$ is equivalent to the one of quasi-categorical right fibrations over $i_1^* B^\sharp$.
\end{rem}

\subsection{Terminal objects and the generalized Yoneda lemma}
\begin{defn}
Let $X$ and $Y$ be simplicial spaces. The (derived) internal hom construction $X^Y$ is the simplicial space with $n$-simplices given by 
\[
\RR \map(\Delta[n] \times Y, X)
\]
where $\RR \map$ denotes the derived mapping space in the levelwise model structure. Henceforth, this will be the standard interpretation, unless mentioned otherwise. To make sure this defines a functor in the variable $[n]$, one should fix cofibrant/fibrant replacement functors in $\SSp$.
\end{defn}

\begin{rem}
The construction above \emph{may not be} the derived functor (in a model categorical sense) of the usual internal hom functor in simplicial spaces, e.g. if the projective model structure is taken. If one works with the injective model structure as Rezk does, then it is. That is a consequence of the fact that this model structure is \emph{closed} (see \cite{Rezk}). We'll be cavalier about this point and simply use the definition above. The reason for that is we will often want to use the projective model structure.
\end{rem}

\begin{defn}[Generalized overcategory]
Let $X$ be a Segal space. For a space $K$ and map $\alpha: K \to X_0$, define $X/K$ to be the simplicial space obtained as the homotopy pullback of
\[
X^{\Delta[1]} \xrightarrow{d_1} X \xleftarrow{\alpha} \Delta[0] \times K
\]
\end{defn}

\begin{rem}
When $K = X_0$ and $\alpha$ is the identity map, $X/X_0$ is the Segal space whose objects are morphisms $x \to y$ in $X$. A morphism is a string $x \to x^{\prime} \to y$ (its source is the composite $x \to y$ and its target is $x^{\prime} \to y$). Note that this construction is not what is often referred to as the arrow category. Rather, it is a overcategory-type construction, as exemplified by the case $K = *$ and $\alpha$ the map selecting an object $x$. 
\end{rem}

\begin{prop}
The map $d_0 : X/K \to X$, induced by $d_0 : X^{\Delta[1]} \to X$, is a right fibration.
\end{prop}
\begin{proof}
Rezk shows \cite{Rezk} that $X^{\Delta[1]}$ is a Segal space. This implies that $X/K$ is a Segal space, so we only need to verify the fibration condition (\ref{eq:rfibsquare}). We will do this for $Y:=X/X_0$. The general case then follows, using that $X/K \simeq Y \times_{X_0} K$.

 By construction, $Y_0$ is identified with $X_1$. Let us describe $Y_1$ in words, first. It is  the space of commutative squares in $X$, i.e. $g f = g^{\prime} f^{\prime}$, with the extra condition that $g^{\prime}$ is an identity morphism. This last condition obviously forces the equality $f^{\prime} = gf$. That is, $Y_1$ is identified with the space of two composable morphisms in $X$, alias $X_2$. More formally, use the decomposition \[\Delta[1] \times \Delta[1] = \Delta[2] \amalg_{\Delta[1]} \Delta[2]\] of the square in two triangles, to write $Y_1$ as the homotopy limit of
\[
X_2 \xrightarrow{d_1} X_1 \xleftarrow{d_1} X_2 \xrightarrow{d_2} X_1 \xleftarrow{s_0} X_0
\]
Since $X$ is a Segal space, we have that $X_2 \times_{X_1} X_0 \simeq X_1$ and hence $Y_1 \simeq X_2$, as claimed. 
\end{proof}

\begin{defn}
An object $x \in X$ is \emph{terminal} if $d_0 : X/x \to X$ is a weak equivalence of simplicial spaces.
\end{defn}

This is equivalent to the condition that the homotopy fiber of the target map $d_1 : X_1 \to X_0$ over $x$ be weakly equivalent to $X_0$. 

\medskip
The crucial observation is

\begin{lem}\label{lem:terminal}
Let $p : X \to B$ be a simplicial space over $B$ and suppose $X$ has a terminal object $x$. Then for every right fibration $q: Y \to B$ the map
\begin{equation}\label{eq:terminal}
x^* : \RR \map_{B}(X, Y) \rightarrow \RR \map_{B}(\Delta[0], Y)
\end{equation}
given by precomposition with $x : \Delta[0] \to X$ is a weak equivalence. That is to say, the inclusion $\Delta[0] \xrightarrow{x} X$ (over $B$) is a weak equivalence in the right fibration model structure.
\end{lem}

The following immediate corollary may be regarded as an ($\infty$,1)-categorical form of the Yoneda lemma.
\begin{lem}[Yoneda lemma]\label{lem:yoneda}
The inclusion
$
\Delta[0] \to Z/z
$
over $Z$ which selects $id_z := s_0(z)$ is a weak equivalence in the right fibration model structure. 
\end{lem}

We will use the following simple observation repeatedly in the proof below. Suppose $B$ is a fibrant simplicial space. By definition, the derived mapping space $\RR \map_{B}(X,Y)$  is weakly equivalent to $\map_{B}(X^c, Y^f)$ where $Y^f \to B$ is a fibrant replacement of $Y \to B$ (a levelwise fibration) and $X^c$ a cofibrant replacement of $X$ (cofibrant as a simplicial space). The latter space is of course the pullback (alias fiber) of
\[
* \to \map(X^c, B) \gets \map(X^c, Y^f)
\]
where the right-hand map is given by composition with $Y^f \to B$. The right hand-map is a fibration, and both mapping spaces in the display are weakly equivalent to the derived mapping spaces. Therefore, for a map $Y \to B$, $\RR \map_{B}(X,Y)$ is naturally weakly equivalent to the homotopy pullback of
\[
* \to \RR \map(X, B) \gets \RR \map(X, Y) \; .
\]

\begin{proof}[Proof of lemma \ref{lem:terminal}]
Without loss of generality, we may assume that $Y \to B$ is fibrant (a right fibration and fibration of simplicial spaces) and $X$ is cofibrant and $B$ is fibrant as a simplicial space. Let $y \in Y_0$ with $q(y) = b$. Define $\RR \map((X,x), (Y,y))$ as the homotopy fiber of the map
\[
\RR \map(X,Y) \xrightarrow{x^*} \RR \map(\Delta[0], Y) = Y_0
\]
over $y \in Y$, and define $\RR \map((X,x), (B,b))$ similarly. By commuting homotopy limits, the homotopy fiber of (\ref{eq:terminal}) over $y$ is identified with the homotopy fiber of 
\[
\RR \map((X,x), (Y,y)) \to \RR \map((X,x), (B,b))
\]
over $p : X \to B$.

 We have a commutative square
\begin{equation}\label{eq:ydiag}
	\begin{tikzpicture}[descr/.style={fill=white}] 
	\matrix(m)[matrix of math nodes, row sep=2.5em, column sep=2.5em, 
	text height=1.5ex, text depth=0.25ex] 
	{
	\RR \map((X,x), (Y,y)) & \RR \map((X/x,id_x), (Y/y,id_y)) \\
	\RR \map((X,x), (B,b)) & \RR \map((X/x,id_x), (B/b,id_b)) \\
	}; 
	\path[->,font=\scriptsize] 
		(m-1-1) edge node [auto] {$$} (m-1-2);
	\path[->,font=\scriptsize] 
		(m-2-1) edge node [auto] {$$} (m-2-2);
	\path[->,font=\scriptsize] 
		(m-1-1) edge node [left] {} (m-2-1);
	\path[->,font=\scriptsize] 		
		(m-1-2) edge node [auto] {} (m-2-2);
	\end{tikzpicture} 
\end{equation}
(Here we follow Rezk and write $id_z \in Z_1$ for the image under $s_0$ of $z \in Z_0$.)

The horizontal maps are weak equivalences because $x$ is the terminal object of $X$. In more detail, the composition
\[
\RR \map((X,x), (Y,y)) \to \RR \map((X/x,id_x), (Y/y,id_y)) \xrightarrow{d_0} \RR \map((X/x,id_x), (Y,y)) 
\]
agrees with precomposition with $d_0 : X/x \to X$ and so it is a weak equivalence
(because $x$ is the terminal object of $X$). This shows that the top horizontal map in (\ref{eq:ydiag}) has a homotopy left inverse.

 Now let $X/id_{x}$ denote the over category $(X/x)/id_x$ whose objects are $2$-morphisms in $X$ with source $id_x$. The composition
\[
	\begin{tikzpicture}[descr/.style={fill=white}] 
	\matrix(m)[matrix of math nodes, row sep=2.5em, column sep=2.5em, 
	text height=1.5ex, text depth=0.25ex] 
	{
	\RR \map((X/x,id_x), (Y/y,id_y)) & \RR \map((X/x,id_x), (Y,y)) \\
	 & \RR \map((X/id_{x},id_{id_x}), (Y/y,id_y)) \\
	}; 
	\path[->,font=\scriptsize] 
		(m-1-1) edge node [auto] {} (m-1-2);
	\path[->,font=\scriptsize] 
		(m-1-2) edge node [auto] {$d_0$} (m-2-2);
	\end{tikzpicture} 
\]
agrees with precomposition with $d_0 : X/id_x \to X/x$ and so is a weak equivalence (because $id_x$ is the terminal object of $X/x$). This shows that the top horizontal map in (\ref{eq:ydiag}) also has a homotopy right inverse, and so is a weak homotopy equivalence. The same arguments show that the lower horizontal map is also a weak equivalence.

 It remains to show that the right-hand map in (\ref{eq:ydiag}) is a weak equivalence. Since $Y \to B$ is a right fibration, the induced map between horizontal homotopy fibers (over $y \in Y_0$ and $b \in B_0$) in the diagram
\[
	\begin{tikzpicture}[descr/.style={fill=white}] 
	\matrix(m)[matrix of math nodes, row sep=2.5em, column sep=2.5em, 
	text height=1.5ex, text depth=0.25ex] 
	{
	Y_0 & Y_1 \\
	B_0 & B_1 \\
	}; 
	\path[<-,font=\scriptsize] 
		(m-1-1) edge node [auto] {$d_1$} (m-1-2);
	\path[<-,font=\scriptsize] 
		(m-2-1) edge node [auto] {$d_1$} (m-2-2);
	\path[->,font=\scriptsize] 
		(m-1-1) edge node [left] {} (m-2-1);
	\path[->,font=\scriptsize] 		
		(m-1-2) edge node [auto] {} (m-2-2);
	\end{tikzpicture} 
\]
is a weak equivalence. That is, $Y/y \xrightarrow{\simeq} B/b$. This finishes the proof.
\end{proof}

\begin{rem}\label{lem:yonedaK}
There are parametrized versions of lemmas \ref{lem:terminal} and \ref{lem:yoneda}. Let $X$ be a Segal space over $B$. If $d_0 : X/K \to X$ is a weak equivalence (i.e. $K$ is a terminal subspace of the space objects of $X$) then
\[
\alpha : K \times \Delta[0] \to X
\]
(over $B$) is a weak equivalence in the right fibration model structure on $\SSp_{/B}$. Therefore, the map $K \times \Delta[0] \to X/K$ is a weak equivalence in the right fibration model structure on $\SSp_{/X}$.
\end{rem}

\subsection{Proof of theorem $A$}

We begin with
\begin{prop}\label{prop:leftadjointGC}
The Grothendieck construction $\sG$ has a simplicial left adjoint
\[
L : \SSp_{/N\sC} \leftrightarrows \Fun(\sC^{op}, \Sp) : \sG
\]
This adjunction is Quillen when both sides are equipped with the projective model structure.
\end{prop}
\begin{proof}
Suppose $X$ is a simplicial space over $N \sC$. Then $X$ can be presented as the colimit (alias coequalizer) in $\SSp$ of
\begin{equation}\label{eq:SSpcoend}
\coprod_{[n_0]} {\Delta}[{n_0}] \times X_{n_0}\leftleftarrows \coprod_{[n_0] \to [n_1] \in \Delta} {\Delta}[{n_0}] \times X_{n_1}
\end{equation}
This is a well known consequence of the Yoneda lemma. Note that the map from this colimit to $X$ is an isomorphism over $N\sC$.

 Left adjoints preserve colimits, and so $L$ is uniquely determined (up to isomorphism) by its value on objects of the form $\alpha : \Delta[n] \times K \to N\sC$ for some space $K$. Suppose given one such object corresponding via the Yoneda lemma to a map $\alpha : K \to N\sC_n$ from $K$ to the space of $n$ composable morphisms in $\sC$. Denote by $\alpha^\prime$ the map $K \to N\sC_n \to N\sC_0$ given by applying $\alpha$ and then $d_1d_2d_3\dots d_n$, the operation selecting the ultimate target object.

 Define
\[
L(\Delta[n] \times K \xrightarrow{\alpha} N\sC) := \coprod_{c \in im(\alpha^{\prime})} K_c \times  \map_{\sC}(-,c)
\]
where $K_c$ is the fiber of $\alpha^\prime : K \to N\sC_0$ over the object $c$ in the image of $\alpha^{\prime}$. Note that when $K = *$, $L(\alpha : \Delta[n] \xrightarrow{} N\sC)$ is the representable functor $\map_{\sC}(-,c_0)$ where $c_0$ is the ultimate target object of $\alpha = (c_0 \gets \dots \gets c_n)$.

 It is a mechanical verification using the Yoneda lemma that
\[
\Hom( L(\alpha), F) \cong \map_{N\sC}(K \times \Delta[n], \sG(F))
\]
for every functor $F$. This settles the claim that $(L,\sG)$ form a simplicial adjunction.
Clearly, $\sG$ preserves fibrations and weak equivalences, and so the pair is Quillen.
\end{proof}

\begin{rem}
In fact, $\sG$ also has a right adjoint. Let $c$ be an object of $\sC$ and $\sC/{c}$ denote the overcategory (this is an internal category, with \emph{space} of objects consisting of maps with target $c$). Let $N(\sC/c) \rightarrow N(\sC)$ be the map induced by the forgetful functor $\sC/c \rightarrow \sC$. The functor $R : \SSp_{/N\sC} \to \PSh(\sC)$ which to $X \in \SSp_{/N \sC}$ associates the presheaf $R(X)$ defined by
\[
R(X)(c) := \map_{N\sC}(N(\sC/c),X)
\]
is the (simplicial) right adjoint to $\sG$.

\end{rem}

The substance of theorem A is encompassed in the following step.
\begin{lem}\label{lem:unitGC}

For each object $X \to N\sC$, the derived unit map $X \to \RR \sG \LL L(X)$ is a weak equivalence in the right fibration model structure on the category of simplicial spaces over $N\sC$.
\end{lem}

\begin{proof}[Proof of lemma \ref{lem:unitGC}]
Since $\sG$ preserves weak equivalences, it is enough to show that the (non-derived) unit map $X \to \sG L(X)$ is a weak equivalence for each cofibrant simplicial space $X$. That is, we are required to show that for each fibrant (i.e. right fibration) $Z$, the induced map
\begin{equation}\label{eq:GCunit}
\RR \map_{N\sC}(\sG L(X), Z) \to \RR \map_{N\sC}(X, Z)
\end{equation}
is a weak equivalence of spaces.

 Let $Y_{\bullet}$ be the canonical resolution of $X$ by representables. Namely, $Y$ is a simplicial object in $\SSp_{/N\sC}$ with simplicial space $Y_k$ of $k$-simplices given by
\[
\coprod_{[n_0] \to [n_1] \to \dots \to [n_k]} X_{n_k} \times \Delta[n_0] 
\]
The Yoneda lemma provides a map $Y_{\bullet} \to X$,  where $X$ is viewed as a constant simplicial object in $\SSp_{/N\sC}$. The induced map $|Y_{\bullet}| \to X$ is well known to be a levelwise weak equivalence in $\SSp_{/N\sC}$ (this can be proved by exhibiting a simplicial contracting homotopy).

 We can thus reduce the problem to the case of representables. By Ken Brown's lemma \cite[7.7.2]{Hirschhorn}, the left Quillen functor $L$ preserves weak equivalences between cofibrant objects. So $L(X) \xrightarrow{\simeq} L(|Y_{\bullet}|)$  since $X$ is cofibrant by hypothesis and $|Y_{\bullet}|$ is the geometric realization of a Reedy cofibrant object, hence cofibrant. Therefore,
\[
\RR \map_{N\sC}(\sG L(X), Z) \simeq \RR \map_{N \sC}( | \sG L(Y_{\bullet}) |, Z) \simeq \textup{Tot} \, \map_{N \sC}(\sG L(Y_{\bullet}) ,Z)
\]
The first equivalence is a consequence of the fact that both $L$ and $\sG$ commute with geometric realization (and colimits in general, since both are left adjoints). In the second equivalence, we have used  that $| \sG L(Y_{\bullet}) |$ is cofibrant since it is the geometric realization of a Reedy cofibrant object. Similarly, $\RR \map_{N\sC}(X, Z) \simeq \textup{Tot} \, \map_{N \sC}(Y_{\bullet} ,Z)$. 

 The upshot is that it suffices to check that (\ref{eq:GCunit}) is a weak equivalence for $X$ of the form $f : \Delta[n] \to N\sC$ corresponding to some string of $n$ morphisms. Because the map $\Delta[0] \to \Delta[n]$ (selecting the first vertex $0$) is a weak equivalence in the right fibration model structure, it suffices to prove the claim for $X = \Delta[0] \to N\sC$ corresponding to an object $b \in N\sC$. This verification occupies the remainder of the proof.

 When $X = (\Delta[0] \xrightarrow{b} N\sC)$, the target of (\ref{eq:GCunit}) is
$\RR \map_{N\sC}(\Delta[0], Z) \simeq Z_b$. We now investigate the source. Recall that $L(f) = \map(-,b)$. Therefore,  $\sG L(f) = N \sC /b$. The inclusion
\[
\Delta[0] \xrightarrow{id_{b}} N \sC /b
\]
over $\{b\} \in N \sC_0$, induces a map
\begin{equation}\label{eq:yon}
\RR \map_{N\sC}(N\sC/b, Z) \to \RR \map_{N \sC}(\{b\},Z) = Z_{b}
\end{equation}
which agrees with $(\ref{eq:GCunit})$. Clearly $\{id_{b} \}$ is a terminal object of $N\sC/b$. This implies, by proposition \ref{lem:terminal}, that (\ref{eq:yon}) is a weak equivalence.
\end{proof}

Moreover, if $X$ is a right fibration over $N \sC$, then we immediately obtain the strictification result that $X$ is levelwise equivalent to an actual (\emph{strict}) presheaf:

\begin{cor}
Suppose $X$ is a right fibration over $N\sC$ and $X$ is cofibrant as a simplicial space. Then the unit map $X \to \sG L(X)$ is \emph{levelwise} weak equivalence.
\end{cor}
\begin{proof}
Weak equivalences between fibrant objects in $\SSp_{/N\sC}$ are determined levelwise, by proposition \ref{prop:rfibmodel}.
\end{proof}

\begin{proof}[Proof of theorem A]
We have already established that $(L, \sG)$ forms a Quillen pair. The unit of the induced adjunction on homotopy categories
\begin{equation}\label{eq:HoGC}
\mathsf{Ho} L : \mathsf{Ho}(\SSp_{/N\sC}) \leftrightarrows \mathsf{Ho}(\Fun(\sC^{\op}, \Sp)) : \mathsf{Ho} \sG
\end{equation}
is an isomorphism by lemma \ref{lem:unitGC}. Moreover, $\sG$ reflects weak equivalences between fibrant objects. It follows that the counit is also an isomorphism.
\end{proof}

\subsection{Explicit description of the left adjoint}

\begin{prop} There is a weak equivalence
\[
\hocolimsub{\Delta} N(c/\sC) \times_{N\sC} X \to L(X)(c)
\]
natural in $c \in \sC$. Here $c/\sC$ is the internal (under)category with space of objects consisting of maps in $\sC$ with source $c$. The left-hand side is clearly functorial in $\sC$ since a map $c \to d$ in $\sC$ induces a map of nerves $N(d/\sC) \rightarrow N(c/\sC)$.
\end{prop}
\begin{proof}
It suffices to consider the case when $X = \Delta[n] \to N\sC$.
For each $c \in \sC$, the simplicial space $N(c/\sC) \times_{N\sC} \Delta[n]$ is equivalent to one with $k$-simplices given by
\[
\coprod_{0 \leq i_0 \leq i_1 \leq \dots \leq i_k \leq n} \map_{\sC}(c, b_{i_0})
\]
where the face maps are given by composition and the degeneracy are obvious. This has a simplicial contraction onto $\map_{\sC}(c,b_n)$, using the maps $b_{i_j} \to b_n$.
\end{proof}

\subsection{Internal version}
There exists a variant of theorem A for category objects in $\Sp$ (alias internal categories in spaces). 

\begin{defn} Let $\sC$ be an internal category in spaces. An internal (contravariant) functor $F : \sC^{op} \to \Sp$ 
(also called \emph{discrete fibration}) consists of an internal category $\sF$ together with a functor $\sF \to \sC$ of internal categories (i.e., a simplicial map between nerves) with the property that the square 
\[
	\begin{tikzpicture}[descr/.style={fill=white}] 
	\matrix(m)[matrix of math nodes, row sep=2.5em, column sep=2.5em, 
	text height=1.5ex, text depth=0.25ex] 
	{
	\sF_1 & \sF_0 \\
	\sC_1 & \sC_0 \\
	}; 
	\path[->,font=\scriptsize] 
		(m-1-1) edge node [auto] {$\mbox{target}$} (m-1-2);
	\path[->,font=\scriptsize] 
		(m-2-1) edge node [auto] {$\mbox{target}$} (m-2-2);
	\path[->,font=\scriptsize] 
		(m-1-1) edge node [left] {} (m-2-1);
	\path[->,font=\scriptsize] 		
		(m-1-2) edge node [auto] {} (m-2-2);
	\end{tikzpicture} 
\]
is cartesian (alias a pullback). The subscript $1$ stands for space of morphisms and the subscript $0$ stands for space of objects. 
\end{defn}

 We adopt the same notation as before and denote by $\Fun(\sC^{op}, \Sp)$ the (internal) category of internal functors $\sC^{op} \to \Sp$. Horel has proved in \cite{Horel} -- under the relatively harmless condition that the target map $\sC_1 \to \sC_0$ is a fibration in $\Sp$ -- that the category $\Fun(\sC^{op}, \Sp)$ has a projective model structure in which a natural transformation $\sF \to \sF^{\prime}$ over $\sC$ is a weak equivalence (resp. fibration) if the map $\sF_0 \to \sF^{\prime}_0$ is a weak equivalence (resp. fibration) in $\Sp$. 

\medskip
The promised mild extension of theorem $A$ is
\begin{thm} Let $\sC$ be an internal category in $\Sp$ such that the target map $\sC_1 \to \sC_0$ is a fibration. Then
\[
L : \SSp_{/N\sC} \leftrightarrows \Fun(\sC^{op}, \Sp) : \sG
\]
is a simplicial Quillen equivalence, where the left-hand side is equipped with the right fibration model structure. Here $\sG$ assigns to an internal functor $\sF \to \sC$ the map of simplicial spaces $N\sF \to N\sC$ given by the nerve.
\end{thm}
\begin{proof}
We need to modify slightly the construction of the left adjoint $L$ from the proof of proposition \ref{prop:leftadjointGC}. For an object $\alpha : \Delta[n] \times K \to N\sC$, declare $L(\alpha)$ to be the category with object space $L(\alpha)_0$ \,:% and morphism space $L(\alpha)_1$ defined as follows:
\[
L(\alpha)_0 := \textup{pullback} (\sC_1 \xrightarrow{\textup{target}} \sC_0 \gets K)
\]
where the right-hand map is the composition $K \to \sC_n \to \sC_0$ of $\alpha$ with the ultimate target operator. The morphism space $L(\alpha)_1$ is defined as $\sC_1 \times_{\sC_0} L(\alpha)_0$.  The source map
\[
L(\alpha)_1 \to L(\alpha)_0
\]
is induced by the composition map $\sC_1 \times_{\sC_0} \sC_1 \to \sC_1$ in $\sC$. The target map is the projection. The forgetful map $L(\alpha) \to N\sC$ gives an internal functor $\sC^{op} \to \Sp$. We leave to the reader the verification that $L$ as prescribed above (and extended to all simplicial spaces by colimits) is the left adjoint to $\sG$.

 The condition that the target map $\sC_1 \to \sC_0$ is a fibration guarantees that $\sG$ preserves fibrant objects and, more generally, fibrations. It is equally clear that $\sG$ preserves weak equivalences. So the pair $(L, \sG)$ is Quillen.

 The argument for the enriched case goes through mostly without substantial changes to show that $(L, \sG)$ is a Quillen equivalence. The key step is to show that, for an object $\alpha : \Delta[0] \times K \to N\sC$ with $K$ a cofibrant space, the unit map
\[
\alpha \to \sG(L(\alpha))
\]
is a weak equivalence in the right fibration model structure on $\SSp_{/N\sC}$. In other words, that for every right fibration $Z \to N\sC$, the induced map
\[
\RR \map_{N\sC}(\sG(L(\alpha)), Z) \to \RR \map_{N\sC}(\alpha, Z)
\]
is a weak equivalence. This can be proved as before; but now both sides are weakly equivalent to $\RR\map_{\sC_0}(K, Z_0)$. For the left-hand side, this follows from lemma \ref{lem:yoneda} and remark \ref{lem:yonedaK}.
\end{proof}

\section{Segal objects in a $\sV$-category}\label{sec:general}

 In this section, we discuss a generalization of the previous constructions, by considering Segal objects in a category $\sV$.

 To conform with model categorical setups, we need to impose a set of technical conditions on $\sV$. We assume that $\sV$ is a monoidal category with respect to the cartesian product and that it is a cartesian closed model category (see e.g. \cite[Ch. 4]{Hovey}, \cite[2.5]{Rezk}). We denote the unit (i.e. the terminal object) by $*$, and we assume it is a cofibrant object.  Finally, we assume that $\sV$ is left proper and combinatorial.

The category $s\sV$ of simplicial objects in $\sV$ is tensored and cotensored over $\sV$. For $X \in s\sV$ and $K \in \sV$, the tensor $X \times K$ is given by $[n] \mapsto X_n \times K$, and the cotensor $X^K$ is given by $[n] \mapsto \map(K,X_n)$, where $\map$ here refers to the internal mapping object in $\sV$. It follows that $s\sV$ becomes a $\sV$-model category (see \cite[4.2.18]{Hovey} for a definition) for the Reedy model structure, and is left proper and combinatorial (from the assumptions we made on $\sV$).
\medskip

 The obvious definition works:

\begin{defn} A \emph{Segal object} in $\sV$ is a functor $X : \Delta^{op} \to \sV$ such that
\[
X_n \to X_1 \times^{h}_{X_0} \dots \times^{h}_{X_0} X_1
\]
is a weak equivalence in $\sV$.
\end{defn}

\begin{defn} Let $Z$ be a Segal object in $\sV$. A map $X \to Z$ in $s\sV$ is said to be a \textbf{right fibration} if $X$ is a Segal object in $\sV$ and the square analogous to (\ref{eq:rfibsquare}) is homotopy cartesian.
\end{defn}

 The technical assumptions on $\sV$ made above guarantee the existence of enriched left Bousfield localizations (see \cite{Barwick}). So we can define a model category structure on $s\sV_{/B}$, the category of simplicial objects in $\sV$ over $B$, whose homotopy theory is that of right fibrations over $B$. 

\begin{prop}\label{prop:rfibmodelV}
Fix a Segal object $B$. There is a left proper, $\sV$-model structure structure on $s\sV_{/B}$ -- called the \textbf{right fibration model structure} -- uniquely characterized by the following properties:
\begin{itemize}
\item An object $X \to B$ is fibrant if it is a right fibration and a fibration in $s\sV$
\item A morphism $X \to Y$ between right fibrations over $B$ is a weak equivalence if $X_0 \to Y_0$ is a weak equivalence
\item A morphism is a cofibration if it is a cofibration in $s\sV$.
\end{itemize}
More generally, a morphism $f : X \to Y$ between any two simplicial objects over $B$ is a weak equivalence if the induced map on derived mapping objects
\[
\RR \map_B(f, W) : \RR \map_B(Y, W) \to \RR \map_B(X, W)
\]
(computed with respect to the levelwise model structure) is a weak equivalence in $\sV$ for every fibrant object $W \to B$.
\end{prop}

We note that, by the assumptions made on $\sV$, the internal mapping object $\map(-,-)$ is a Quillen bifunctor, and so it admits a derived functor $\RR \map$.

\section{Cartesian fibrations}\label{sec:cartfib}
In this section, we take $\sV$ to be Rezk's category of complete Segal spaces $\Caty$.  A Segal object in $\Caty$ is what one might call a double Segal space, that is, a bisimplicial space $X_{m,n}$ such that $X_{\bullet,n}$ is a Segal space (for each fixed $n \geq 0$) and $X_{m,\bullet}$ is a Segal space (for each fixed $m \geq 0$). This is reminiscent of a double category (a category object in categories) not of a $2$-category.

 A Segal space $Z$ is an example of a double Segal space: set $X_{m,n} = Z_m$ for $m,n \geq 0$.

\begin{rem} A double Segal space as defined is not a fibrant object in $s\sV$, the missing condition being that $X_{\bullet,n}$ is complete as a Segal space.
\end{rem}

For the rest of this section, $B$ denotes a given double Segal space.

\begin{defn}\label{defn:cartfib}
A map $X \to B$ of bisimplicial spaces is a \textbf{Cartesian fibration} if $X$ is a double Segal space and the map $$X_{\bullet,n} \to B_{\bullet,n}$$ is a right fibration (of simplicial spaces) for every $n$.
\end{defn}

\begin{rem}\label{rem:cartfib}
A word of warning about terminology: the definition above is not the literal analogue of the classical definition of a Cartesian fibration (or its quasi-categorical version), so it may be argued that it's not fully deserving of that name. For example, definition \ref{defn:cartfib} is one categorical level higher when compared to the standard definition. We decided to keep the potentially conflicting terminology, in part due to theorem B (both perspectives give weak models for functors with values in $\Caty$) and because it is plausible that a direct comparison should exist.
\end{rem}

\subsection{Homotopy theory of Cartesian fibrations}

Because of its importance, we restate proposition \ref{prop:rfibmodelV} for the special case $\sV = \Caty$ below.

\begin{prop}
There is a left proper, $\Caty$-enriched model structure on the category of bisimplicial spaces over $B$ uniquely characterized by the following data.
\begin{itemize}
\item An object $X \to B$ is fibrant if it is a Cartesian fibration and $X_{m,\bullet}$ is a complete Segal space for each $m$.
\item A map $f : X \to Y$ over $B$ is a cofibration if it is so in the underlying model structure on bisimplicial spaces.
\end{itemize}
A map $f : X \to Y$ over $B$ is a weak equivalence if $f_{m, \bullet}$ is a weak equivalence in the complete Segal space model structure for all $m$, and $f_{\bullet, n}$ is a weak equivalence in the right fibration model structure for all $n$.
\end{prop}

\begin{rem}
Suppose $B$ is a Segal space (viewed as a double Segal space). Let $X$ and $Y$ be  over $B$, but $X_{n,\bullet}$ and $Y_{n,\bullet}$ not necessarily complete. A map $f : X \to Y$ over $B$ is a weak equivalence if and only if the map
\[
f_{0,\bullet} : X_{0, \bullet} \to Y_{0, \bullet}
\]
is a Dwyer-Kan equivalence or, equivalently, the map
\[
f_{b,\bullet} : X_{b, \bullet} \to Y_{b, \bullet}
\]
is a Dwyer-Kan equivalence for every $b \in B_0$.
\end{rem}

\subsection{Strictification}

Let $F : \sC^{op} \to s\Sp$ be an simplicially enriched functor. Consider $N\sC$ as a bisimplicial space trivially, i.e. ${N(\sC)}_{m,n} = N(\sC)_m$. Form a bisimplicial space $\sG(F)$ with space of $(m,n)$-simplices given by
$$
\sG(F)_{m, n} = \coprod_{(c_0, \dots, c_m)} \map_{\sC}(c_m, c_{m-1}) \times \dots \times \map_{\sC}(c_{1}, c_{0}) \times F(c_0)_n
$$

We call $\sG$ the \emph{Grothendieck construction}. (We point out that this is not the standard usage of the term when applied to functors with values in categories.) There is an adjunction
\[
L: s\SSp_{/N\sC} \leftrightarrows \Fun(\sC^{op}, s\Sp) : \sG 
\]

It is easy to check that
\[
L(\Delta[m,n] \xrightarrow{f} N\sC) = \Delta[n] \times \map_{\sC}(-,c_0)
\]
where $f$ corresponds to a string of morphisms $c_0 \gets \dots \gets c_m$ in $\sC$.

If $F(c)$ is a Segal space for each $c \in \sC$, then $\sG(F) \to N\sC$ is a Cartesian fibration. The pair $(L,\sG)$ lifts to a Quillen pair between the Cartesian model structure on bisimplicial spaces over $N\sC$ and the projective model structure on $\Fun(\sC^{op}, \Caty)$.

\begin{lem}
The right adjoint $\sG$ preserves weak equivalences.
\end{lem}
\begin{proof}
This follows rather directly from the definitions of a weak equivalence in $\Fun(\sC^{op}, \Caty)$ and the proposition above.
\end{proof}

When $X = \Delta[{m,n}]$, the following is essentially the Yoneda lemma.

\begin{lem}\label{lem:cartfib-unit} For $X$ a cofibrant bisimplicial space and a map $X \to N\sC$, the non-derived unit map
$
X \to \sG L (X)
$ 
(over $N\sC$) is a weak equivalence in the Cartesian fibration model structure.
\end{lem}
\begin{proof}
Let $X$ be a representable $f : \Delta[{m,n}] \to N\sC$ corresponding to a string $c = c_0 \gets \dots \gets c_m $ of morphisms in $\sC$. Then $L(f) = \Delta[n] \times \map_{\sC}(-,c)$ and $\sG L (f)$ is the \emph{perpendicular} (or box) product $N (\sC/c) \Box \Delta[n]$ of $N(\sC/c)$ and $\Delta[n]$, i.e. the bisimplicial space with space $N(\sC/c)_i \times \Delta[k]_j$ of $(i,j)$-simplices.

 The map $\Delta[{0,n}] \to \Delta[{m,n}]$ (over $N\sC$) that selects the object $c$ is a weak equivalence by definition.  The map $\Delta[{0,n}] \to N (\sC/c) \Box \Delta[n]$ over $N\sC$ is a weak equivalence in the Cartesian fibration model structure as we now explain. Suppose $W \to N\sC$ is a Cartesian fibration. By adjunction, we have
$$
\RR \map_{N\sC}(N(\sC/c) \Box \Delta[n], W) \simeq \RR \map_{N\sC}(N(\sC/c), W_{\bullet, n})
$$
But the right-hand mapping space is weakly equivalent to $\RR\map_{N\sC}(\Delta[{0,n}], W)$ since $W_{\bullet,n}$ is a right fibration over $N\sC$ for every $n$ by assumption. This proves the claim for representable objects. For an arbitrary cofibrant $X$, resolve by representables.
\end{proof}

\begin{proof}[Proof of theorem B]
Since $\sG$ preserves weak equivalences, the derived unit is weakly equivalent to the the non-derived unit for cofibrant objects, which is a weak equivalence by lemma \ref{lem:cartfib-unit}. Moreover, the Grothendieck construction $\sG$ reflects weak equivalences between fibrant objects. Both $ss\Sp_{/N\sC}$ and $\Fun(\sC^{op}, \Caty)$ are enriched model categories over $\Caty$ and the pair $(L,\sG)$ is a $\Caty$-enriched adjunction. Hence, the derived adjunction $(\LL L, \RR \sG)$ is an equivalence of $(\infty,2)$-categories. 
\end{proof}

\begin{rem}
There is a variant of theorem B when $\sC$ is a bicategory.
\end{rem}

\section{(Co)limits for a general $\sV$}\label{sec:colim}

\subsection{Colimits}

Let $C$ be a Segal object in $\sV$ and $p : X \to C$ a right fibration. We define the \emph{homotopy colimit} of $p$ as the homotopy colimit of $X : \Delta^{op} \to \sV$. This is justified by the following.

\begin{prop}\label{prop:hocolim} The homotopy colimit of $p$ is the homotopy left adjoint to the constant diagram functor
\[
s\sV \to s\sV_{/ C}
\]
which sends $Y$ to $Y \times C \to C$. Here both categories are given the right fibration model structure (with underlying projective model structure on $\sV$).
\end{prop}
\begin{proof}
Write $t : C \to *$ for the terminal map and $t^*$ for the constant diagram functor. Then $t^*$ has a left adjoint $t_!$ given by $t_!(X \to C) = X$ and $(t_!, t^*)$ form a Quillen pair for the projective model structure on $s\sV$. Because $t^*$ preserves right fibrations, this pair is also Quillen with respect to the right fibration model structures. Note that, for the right fibration model structure on $s\sV = s\sV_{/*}$, fibrant objects are the homotopically constant simplicial objects (i.e. all maps are weak equivalences).

 Write $c : \sV \to s\sV$ for the constant diagram functor. It remains to show that $t_! (p) = X$ is weakly equivalent to $c(\hocolim X)$ in the right fibration model structure on $s\sV$. For a fibrant (i.e. homotopically constant) simplicial object $Z$, we have that $c(\hocolim Z) \simeq Z$. So the result follows from the chain of weak equivalences:
\[
\begin{array}{rcl}
 \RR \map(X, Z)  &  \simeq & \RR \map (X, c(\hocolim Z))   \\
  &  \simeq &   \RR \map (\hocolim X, \hocolim Z) \\
  &  \simeq & \RR \map(c(\hocolim X), c(\hocolim Z)) \\
\end{array}
\]
where the second equivalence comes from the adjunction between $\hocolim$ and $c$.
\end{proof}

\subsection{Limits}

If $\sV$ is a cartesian closed model category, then $s\sV$ is a cartesian closed model category with the injective model structure. It follows from the definition that, for every cofibrant object $B$, the adjunction
\[
- \times B : s\sV \leftrightarrows s\sV : (-)^B
\]
forms a Quillen pair. This implies that (and is in fact equivalent to) the following pair is also Quillen:
\[
t^* : s\sV \leftrightarrows s\sV_{/B} : t_*
\]
where $t^*(X) = X \times B \to B$ and $t_*(Y \to B) = \map_{A}(A, Y)$, viewed as a constant simplicial object. We then have

\begin{prop} Let $B$ be a (cofibrant) Segal object in $\sV$. The constant diagram functor $t^*$ has a right adjoint
\[
t^* : s\sV \leftrightarrows s\sV_{/B} : t_*
\]
and the pair is Quillen with respect to the right fibration model structure (with underlying injective model structure on $s\sV$). Moreover, $\RR t_*$ evaluated at a right fibration $p : Y \to B$ is identified with the derived sections object $\RR \map_{B}(B,Y)$, i.e. the pullback of
\[
* \to \RR \map(B,B) \leftarrow \RR\map(B,Y)
\]
computed in $\sV$. We call $\RR t_*(p)$ the (homotopy) limit of $p$.
\end{prop}

\section{Kan extensions}\label{sec:kan}
 We now deal with Kan extensions of right fibrations with values in spaces. Suppose $i: C \to D$ is a map between Segal spaces. Base change along $i$ induces a functor
\[
i^* : \SSp_{/D} \rightarrow \SSp_{/C}
\]
by setting $i^*(Y \to D) = Y \times_{D} C$.

\begin{prop}[Left Kan extensions exist] Base change along $i$ is the right adjoint in a simplicial Quillen pair
\[
i_! : \SSp_{/C} \leftrightarrows \SSp_{/D} : i^*
\]
where both categories are equipped with the right fibration model structure. The derived left adjoint $\LL i_!$ is called the homotopy left Kan extension along $i$.
\end{prop}

\begin{proof}
The left adjoint $i_!$ is given by post-composition
\[
i_! ( Z \xrightarrow{f} C) = Z \xrightarrow{i \circ f} D 
\]
When $Z = \Delta[n]$, we have
$$
\map_{D}(i_!(f), X) \xrightarrow{\cong} \map_{C}(f, i^*X)
$$
This map is induced, via the Yoneda lemma, by $X_{i(f)} \xrightarrow{\cong} (i^*X)_{f}$. The general case follows.

 By construction, $i_!$ preserves cofibrations since it preserves generating cofibrations. We also know that right fibrations are stable under base change (proposition \ref{prop:rfibbase}) which implies that $i^*$ preserves fibrant objects (i.e. those right fibrations which are also levelwise fibrations).
\end{proof}

The next topic is how to characterize of Kan extensions. Let $i : C \to D$ be a map of Segal spaces. The key construction is a bisimplicial space $D/i$ which in bidegree $(m,n)$ consists of strings of composable morphisms in $\sD$ of the form
\[
i(c_0) \gets \dots \gets i(c_n) = d_0 \gets \dots \gets d_m \gets d
\]
(the last $n$ morphisms come from morphisms in $C$). Functoriality in each simplicial direction is given by composing, or forgetting, morphisms (face maps) or adding identities (degeneracy maps).

 More precisely, let us define
\[
(D/i)_{m,n} := D_{m+n+1} \times_{D_n} C_n
\]
where $D_{m+n+1} \to D_n$ is the map induced by the inclusion $\gamma := d_{n+1} \dots d_{m+n+1} : [n] \to [m+n+1]$, i.e. $\gamma(i) = i$.

 On morphisms, $D/i$ is constructed as follows. Fix $m$. The degeneracy maps in the $n$-direction
\[
s_i : (D/i)_{m,n} \to (D/i)_{m,n+1}
\]
(for $0 \leq i \leq n$) are given by $(s_{i+m+1}, s_i)$. The face maps are defined as $(d_{i+m+1}, d_i)$ (inspection shows that it induces a map of pullbacks). Now fix $n$. Given a map $j : [m] \to [m^{\prime}]$ in $\Delta$, extend by the identity to a map $J : [m + n + 1] \to [m^{\prime} + n + 1]$. That is, 
\[
J(i) = 
\left\{
	\begin{array}{ll}
		i & \mbox{if } i \leq n \\
		j(i-n) & \mbox{if } n \leq i \leq m + n \\
		i-m-n & \mbox{if } i > m + n
	\end{array}
\right.
\]
Then declare
\[
(D/i)_{m^{\prime},n} \to (D/i)_{m,n}
\]
to be the map given by $(J^*, id_{C_n})$. We leave to the reader the verification that $(D/i)_{m,n}$ as defined is indeed a bisimplicial space.

 There is an alternative definition of $D/i$ which highlights functoriality. Namely, for $n \geq 0$, let $(D/i)_m$ be the simplicial space defined as the homotopy limit of the diagram
\[
D_m \times \Delta[0] \to D^{\Delta[m]} \xleftarrow{d_{0}} D^{\Delta[m+1]} \xrightarrow{d_1 \dots d_{m-1}} D \xleftarrow{} C
\]
(recall that $d^1 \dots d^{m-1} : [0] \to [n]$ induces the ultimate target map). This is clearly functorial in $[m]$. Under the assumption that $C$ and $D$ are Segal spaces, it is not hard to see that the two definitions of $D/i$ are equivalent.

\medskip
 Note that, for each $n \geq 0$, there is a map of simplicial spaces 
\begin{equation}\label{eq:mapDi}
(D/i)_{\bullet, n} \to D
\end{equation}
induced by $\beta^* : D_{m+n+1} \to D_m$, where $\beta : [m] \to [m+n+1]$ sends $i$ to  $n+1+i$. With the second description of $D/i$, this map is simply the projection onto the $D_m \times \Delta[0]$ factor.

\begin{lem}\label{lem:lkanfib}
For each $n \geq 0$, (\ref{eq:mapDi}) is a right fibration.
\end{lem}
\begin{proof}
Since $D$ is a Segal space, we have that $D_m \times_{D_0} \times D_{n+1} \simeq D_{m+n+1}$, and so it follows that $D_m \times_{D_0} (D/i)_{0,n} \simeq (D/i)_{m,n}$.
\end{proof}

\begin{prop}\label{prop:lkan}
Let $X$ be a right fibration over $C$. The simplicial space given by
\[
[m] \mapsto \hocolimsub{[n] \in \Delta} (D{/i})_{m,n} \times^h_{C_n} X_n
\] 
is a right fibration over $D$ and computes $\LL i_!(X)$, the homotopy left Kan extension of $X$ along $i$.
\end{prop}
\begin{proof}
For the duration of this proof, we will sometimes write $\sL$ for the simplicial space over $D$ in the statement of the proposition. Because in $\Sp$ homotopy colimits are stable under homotopy base change, we have that
\[
\bigl(\hocolimsub{\Delta} (D{/i})_{0} \times^h_{C} X \bigr) \times^h_{D_0} D_m \simeq \hocolimsub{\Delta} \bigl((D{/i})_{0} \times^h_{D_0} D_m \times^h_{C} X \bigr)
\]
where we have suppressed $n$ from the notation.
By lemma \ref{lem:lkanfib}, the target is weakly equivalent to 
\[
\hocolimsub{\Delta} (D{/i})_{m} \times^h_C X = \sL_{m}
\]
and so, by proposition \ref{prop:rightfibequiv}, we conclude that $\sL \to D$ is a right fibration.

\medskip
 The second claim is that there is a natural zigzag of weak equivalences relating $\LL i_! (X)$ and $\sL$ in the right fibration model structure.

 We may reduce statement to representables, i.e. $X$ of the form $C/K \to C$ for $K \subset C_0$. We take $K = C_0$ (the general case is analogous). In that case, $\LL i_! (X)$ is identified with $i_0 : \Delta[0] \times C_0 \rightarrow{} D$ -- first, using the Yoneda lemma, we replace $X$ with $\Delta[0] \times C_0 \xrightarrow{} C$, and then we apply (non-derived) $i_!$.

 We are left to identify $i_0 : \Delta[0] \times C_0 \rightarrow{} D$ and $\sL$. Applying the Yoneda lemma again, we identify $i_0 : \Delta[0] \times C_0 \rightarrow{} D$ with
\[
D/C_0 \to D
\]
Note that $(D/C_0)_m = (D/i)_{m,0}$. In the remainder of the proof, we will construct a weak equivalence\footnote{This can be regarded as an analogue of the Yoneda lemma in the following formulation:
\[
\map_{\sD}(-,i(-)) \otimes_{\sC} \map_{\sC}(-,c) \cong \map_{\sD}(-,i(c))
\]
}
\[
\hocolimsub{[n] \in \Delta}(D/i)_{m,n} \times_{C_n} C_{n+1} \simeq (D/i)_{m,0}
\]
The claimed weak equivalence will be in fact given as a zigzag of weak equivalences.

As a first step, we can replace the hocolim in the display above with
\begin{equation}\label{eq:saug}
\hocolimsub{[n] \in \Delta} (D/i)_{m,n+1} 
\end{equation}
Indeed, there is a map of simplicial objects
\begin{equation}\label{eq:saugm}
(D/i)_{m,n+1} \to (D/i)_{m,n} \times_{C_n} C_{n+1}
\end{equation}
induced by the map between homotopy pullbacks of the rows in the diagram
\[
	\begin{tikzpicture}[descr/.style={fill=white}] 
	\matrix(m)[matrix of math nodes, row sep=2.5em, column sep=2.5em, 
	text height=1.5ex, text depth=0.25ex] 
	{
	D_{m+n+2} & D_{n+1} & C_{n+1}\\
	D_{m+n+1} & D_n & C_{n+1} \\
	}; 
	\path[->,font=\scriptsize] 
		(m-1-1) edge node [auto] {$\gamma^*$} (m-1-2);
	\path[->,font=\scriptsize] 
		(m-2-1) edge node [auto] {$\gamma^*$} (m-2-2);
	\path[->,font=\scriptsize] 
		(m-1-1) edge node [left] {$d_0$} (m-2-1);
	\path[->,font=\scriptsize] 		
		(m-1-2) edge node [auto] {$d_0$} (m-2-2);
	\path[->,font=\scriptsize] 
		(m-1-3) edge node [above] {$p$} (m-1-2);
	\path[->,font=\scriptsize] 		
		(m-1-3) edge node [auto] {$id$} (m-2-3);
	\path[->,font=\scriptsize] 		
		(m-2-3) edge node [above] {$d_0p$} (m-2-2);

	\end{tikzpicture} 
\]
(recall $\gamma$ is the map $\gamma(i) = i$, wherever defined). But the left-hand square in the diagram is homotopy cartesian since $D$ is a Segal space, and it follows that the induced map on homotopy pullbacks is a weak equivalence. Thus, $(\ref{eq:saugm})$ is a levelwise weak equivalence of bisimplicial spaces.

 Now, the face operator $d_1 : D_{m+2} \to D_{m+1}$ (which corresponds to the composition of the last two morphisms) yields a map
\[
(D/i)_{m,1} \to (D/i)_{m,0} 
\]
which is an augmentation for $[n] \mapsto (D/i)_{m,n+1}$, and thus gives rise to a map 
from (\ref{eq:saug}) to $(D/i)_{m,0}$. It remains to see that this map is a levelwise weak equivalence of simplicial spaces (in the variable $m$). Since both simplicial spaces are right fibrations over $D$, it is enough to show that it is a weak equivalence on $0$-simplices, i.e. for $m = 0$. This can be done by exhibiting a contracting homotopy using the extra degeneracy maps provided by $s_{0} : D_k \to D_{k+1}$.
\end{proof}

Setting $D = *$ in proposition \ref{prop:lkan}, we recover proposition \ref{prop:hocolim} with $\sV = \Sp$.

\begin{rem}
When $i$ is the nerve of a functor of $\Sp$-categories and $F$ is a presheaf on $C$, we can set $X = \sG(F)$ in proposition \ref{prop:lkan} to recover the classical formulae computing the homotopy left Kan extension of $F$ along $i$. In particular, when $D$ is a point, $\LL i_! \sG(F)$ computes the usual homotopy colimit of $F$ as in \cite{BousfieldKan}.
\end{rem}

\subsection{Invariance under $\infty$-categorical equivalences}\label{sec:inv}

\begin{prop}\label{prop:DKequivCD}
If $i : C \to D$ is a Dwyer-Kan equivalence between Segal spaces, then
\[
i_! : \SSp_{/C} \leftrightarrows \SSp_{/D} : i^*
\]
is a simplicial Quillen equivalence, where both sides have the right fibration model structure.
\end{prop}

\begin{proof}
We show that the derived unit is a weak equivalence and that $\RR i^*$ is homotopy conservative. Let $X$ be a simplicial space over $C$. The derived unit is the morphism (over $C$)
\begin{equation}\label{eq:unitCD}
X \to \RR i^* \LL i_! X
\end{equation}

 We begin with the case $X = \{c\} : \Delta[0] \xrightarrow{c} C$. By the Yoneda lemma (lemma \ref{lem:yoneda}), the vertical maps in the diagram below are weak equivalences in the right fibration model structure
\[
	\begin{tikzpicture}[descr/.style={fill=white}] 
	\matrix(m)[matrix of math nodes, row sep=2.5em, column sep=2.5em, 
	text height=1.5ex, text depth=0.25ex] 
	{
	\{c\} &  \RR i^* \LL i_! (\{c\}) \\
	C/c & \RR i^* \LL i_! (C/c) \\
	}; 
	\path[->,font=\scriptsize] 
		(m-1-1) edge node [auto] {} (m-1-2);
	\path[->,font=\scriptsize] 
		(m-2-1) edge node [auto] {} (m-2-2);
	\path[->,font=\scriptsize] 
		(m-1-1) edge node [left] {$\simeq$} (m-2-1);
	\path[->,font=\scriptsize] 		
		(m-1-2) edge node [auto] {$\simeq$} (m-2-2);
	\end{tikzpicture} 
\]

But $i_! (\{c\}) = D/i(c)$ and $\{c\}$ is cofibrant, and so $\RR i^* \LL i_! (\{c\}) \simeq \RR i^*(D/i(c))$. Moreover, $D/i(c) \to D$ is a right fibration, so $\RR i^*(D/i(c)) \simeq i^*(D/i(c))$. This shows that the top right-hand space in the square is weakly equivalent to $i^*(D/i(c))$.

 Now, by hypothesis $i$ is fully faithful. That is to say, the map $C/c \to i^*(D/i(c))$ is a weak equivalence. This shows that the horizontal maps in the diagram above are weak equivalences.

 When $X$ is of the form $f : \Delta[n] \to C$ corresponding to a string of maps $c_0 \gets \dots \gets c_n$, the claim follows from the previous case by proposition \ref{prop:rightfibequiv}. The general case can be proved by reduction to the previous two cases, in the same spirit as the proofs of the previous propositions. Given a simplicial space $X$ over $C$, resolve it by representables $Y_{\bullet}$. Since $i_!$ and $i^*$ commute with colimits (the latter since it is a left adjoint, the former since colimits in spaces are stable under base change), the map (\ref{eq:unitCD}) is a weak equivalence if
\[
Y_{k} \to \RR i^* \LL i_! Y_{k}
\]
is a weak equivalence for each $k \geq 0$. This concludes the proof that the derived unit is a weak equivalence. 

It remains to prove that $\RR i^*$ is homotopy conservative, i.e. that it detects weak equivalences between fibrant objects. Suppose that $f : X \to Y$ is a map of right fibrations over $D$ and that the induced map $i^*X \to i^*Y$ is a weak equivalence of right fibrations over $C$. This means that for each object $c \in C_0$ the induced map of homotopy fibers $$X_{i(c)} \to Y_{i(c)}$$ is a weak equivalence of spaces. But, since $i$ is essentially surjective, for every $d \in D_0$ there exists some $c \in C_0$ such that $i(c) \simeq d$. Since right fibrations are fiberwise complete (proposition \ref{prop:fibcss}), the spaces $X_{i(c)}$ and $X_d$ are weakly equivalent, and similarly for $Y$. It follows that $f$ is a weak equivalence in degree $0$, and so in every degree.
\end{proof}

\begin{cor}\label{cor:comp}
Suppose $B$ is a Segal space (not necessarily complete) and let $i : B \to {B}^{\sharp}$ be its Rezk completion, i.e. a Dwyer-Kan equivalence to a \emph{complete} Segal space $B^{\sharp}$. Then $i^*$ is a Quillen equivalence between the right fibration model structures on $B$ and $B^{\sharp}$.
\end{cor}

\begin{rem}
Since, by \cite{Bergner}, any Segal space $B$ may be strictified (i.e. is Dwyer-Kan equivalent) to some simplicially enriched category $\sC$, theorem A and the invariance statement above together imply that any right fibration over $B$ may be strictified to a space-valued presheaf on $\sC$.
\end{rem}

\bibliography{gcbib}
\bibliographystyle{abbrv}
\end{document}